\documentclass[a4paper,12pt,reqno]{amsart}

\usepackage[utf8]{inputenc}
\usepackage{amsmath,amssymb,hyperref,graphicx,xcolor}
\usepackage[a4paper,margin=1.9cm,top=2.5cm,bottom=2.5cm,centering,vcentering]{geometry}

\theoremstyle{plain}
\newtheorem{theorem}{Theorem}[section]
\newtheorem{lemma}[theorem]{Lemma}
\newtheorem{proposition}[theorem]{Proposition}
\newtheorem{cor}[theorem]{Corollary}
\theoremstyle{definition}

\newtheorem{exa}[theorem]{Example}

\newtheorem{remark}[theorem]{Remark}
\numberwithin{equation}{section}

\usepackage{amscd}

\begin{document}
\title[Zumkeller numbers and $k$-layered numbers]{Some properties of Zumkeller numbers and $k$-layered numbers}
	\author[P. J. Mahanta]{Pankaj Jyoti Mahanta}
	\address{Gonit Sora, Dhalpur, Assam 784165, India} 
	\email{pankaj@gonitsora.com}
\author[M. P. Saikia]{Manjil P. Saikia}
\address{School of Mathematics, Cardiff University, CF24 4AG, UK}
\email{SaikiaM@cardiff.ac.uk, manjil@gonitsora.com}
\author[D. Yaqubi]{Daniel Yaqubi}
\address{Faculty of Agriculture and Animal Science, University of Torbat-e Jam, Iran}
\email{daniel\_yaqubi@yahoo.es}
    \thanks{The second author is partially supported by the Leverhulme Trust Research Project Grant RPG-2019-083.}

    \newcommand{\manjil}[1]{\textcolor{blue}{#1}}
        \newcommand{\daniel}[1]{\textcolor{red}{#1}}

\begin{abstract}
Generalizing the concept of a perfect number is a Zumkeller or integer perfect number that  was introduced by Zumkeller in 2003. The positive integer $n$ is a Zumkeller number if its divisors can be partitioned into two sets with the same sum, which will be $\sigma(n)/2$. Generalizing even further, we call $n$ a $k$-layered number if its divisors can be partitioned into $k$ sets with equal sum.

In this paper, we completely characterize Zumkeller numbers with two distinct prime factors and give some bounds for prime factorization in case of Zumkeller numbers with more than two distinct prime factors. We also characterize $k$-layered numbers with two distinct prime factors and even $k$-layered numbers with more than two distinct odd prime factors. Some other results concerning these numbers and their relationship with practical numbers and Harmonic mean numbers are also discussed.
\end{abstract}

\maketitle

\section{Introduction}
Let $n$ be a positive integer. The canonical decomposition of $n$ as a product of primes can be
written as \[n=p_1^{\alpha_1}\cdots p_r^{\alpha_r},\] where $p_1<p_2<\cdots <p_r$. As usual we shall denote the \textit{number of divisors} of $n$ by $\tau(n)$ and the \textit{sum of the divisors} by $\sigma(n)$, so that
\[\tau(n)=(\alpha_1+1)(\alpha_2+1)\cdots (\alpha_r+1)\] and
\[\sigma(n)=\sum_{d|n}d=\left(\dfrac{p_1^{\alpha_1+1}-1}{p_1-1}\right)\cdots \left(\dfrac{p_r^{\alpha_r+1}-1}{p_r-1}\right).\]
The function $\sigma(n)$ is frequently studied in connection with perfect numbers. A number $n$ is called a \textit{perfect number} if $\sigma(n)=2n$. Even perfect numbers have been studied and classified since antiquity and we know from the work of Euclid and Euler that $n$ is an even perfect number if $n=2^{p-1}(2^{p}-1)$ where both $p$ and $2^p-1$ are primes (see for instance the book of Hardy and Wright \cite{HardyWright}). However, to date no odd perfect number has been found nor is known to exist.

This tantalizing question of the existence of odd perfect numbers has motivated many mathematicians to look at generalizations which are hoped to shed light on the original problem (for instance, see the second author's joint work with Laugier and Sarmah \cite{LaugierSaikiaSarmah} and with Dutta \cite{DuttaSaikia}, and the references therein). One of these generalization is the concept of a \textit{Zumkeller} or \textit{integer perfect number} that was introduced by Zumkeller in 2003. The positive integer $n$ is a Zumkeller number if its divisors can be partitioned into two sets with the same sum, which will be $\sigma(n)/2$ (see the OEIS sequence \href{https://oeis.org/A083207}{A083207} for a few values). For example, $20$ is a Zumkeller number because its divisors, $1,2,4,5,10$ and $20$ can be partitioned in the two sets
$A = \{1, 20\}$ and $B = \{2, 4, 5, 10\}$ whose common sum is $21$.

Undoubtedly, one of the most  important reasons for studying Zumkeller numbers is because each perfect number is a Zumkeller number. Clearly, the  necessary conditions for a positive integer $n$ to be a Zumkeller number are $\sigma(n)\geq 2n$ and $\sigma(n)$ is even. Therefore, the number of positive odd factors of $n$ must be even. Peng and Bhaskara Rao \cite[Proposition 3]{Pen} proved that the necessary and sufficient condition for a positive integer $n$ to be a Zumkeller number is that the value of $\dfrac{\sigma(n)-2n}{2}$ must be either $0$ or be a sum of distinct positive factors of $n$ excluding $n$ itself. So, we see that every perfect number is Zumkeller.

In 2008 Clark \textit{et. al.} \cite{Cla}, announced several results about Zumkeller, \textit{Half-Zumkeller} and \textit{practical} numbers and made some conjectures. A positive integer $n$ is said to be a Half-Zumkeller number if the proper positive divisors of $n$ can be partitioned into two disjoint parts so that the sums of the two
parts are equal. Also, $n$ is said to be practical if all positive integers
less than $n$ can be represented as sums of distinct factors of $n$. (See the OEIS sequences \href{https://oeis.org/A246198}{A246198 } and \href{https://oeis.org/A005153}{A005153} for the first few values of Half-Zumkeller and practical numbers respectively.) Bhakara Rao and Peng \cite{Pen} provided a proof for the first conjectures of Clark \textit{et. al.} \cite{Cla}, and proved the second conjecture in some special cases. They also proved several other results and posed a few open problems on Zumkeller numbers.

After the work of Bhaskar Rao and Peng \cite{Pen} there seems to be no further work on Zumkeller numbers from a purely number-theoretic perspective until recently when Jokar \cite{Jokar} studied a generalization of Zumkeller numbers called $k$-layered numbers. A natural number $n$ is called a $k$-layered number if its divisors can be partitioned into $k$ sets with equal sum. Clearly, Zumkeller numbers are $2$-layered and hence $k$-layered numbers are a proper generalization of Zumkeller numbers.

The goal of the present paper is to remedy this situation and initiate once again the study of Zumkeller numbers and its generalization. With this aim, we will characterize all Zumkeller numbers with two distinct prime factors in Section \ref{sec:two}, then we characterize all $k$-layered numbers with two distinct prime factors as well as even $k$-layered numbers with two distinct odd prime factors in Section \ref{sec:layer}, in Section \ref{sec:harmonic} we prove several results connecting Zumkeller numbers with harmonic mean numbers as well as prove some results about Zumkeller numbers with more than two distinct prime factors, and finally we end the paper with some concluding remarks in Section \ref{conc}. It is hoped that other mathematicians would find it fruitful to extend our results further.

 \section{Zumkeller numbers with two distinct prime factors}\label{sec:two}

 We have already stated that one of the necessary condition for a positive integer $n$ be a Zumkeller number is $\sigma(n)\geq 2n$. For distinct odd prime numbers $p_1$ and $p_2$, it is impossible for $n=p_1^{\alpha}p_2^{\beta}$ to be a Zumkeller number, for any $\alpha, \beta \in \mathbb{N}$. Since if $n=p_1^{\alpha}p_2^{\beta}$ be a Zumkeller number, then we must have
 \[\sigma(p_1^{\alpha}p_2^{\beta})=\left(\dfrac{p_1^{\alpha+1}-1}{p_1-1}\right)\left(\dfrac{p_2^{\beta+1}-1}{p_2-1}\right)\geq 2p_1^{\alpha}p_2^{\beta},\]
 which leads to
 \[(p_1^{\alpha+1}-1)(p_2^{\beta+1}-1)\geq 2p_1^{\alpha}p_2^{\beta}(p_1-1)(p_2-1).\]
 This gives us
 \begin{eqnarray*}
   p_1^{\alpha}p_2^{\beta}((p_1-2)(p_2-2)-2)+p_1^{\alpha+1}+p_2^{\beta+1}-1\leq 0,
\end{eqnarray*}
which is impossible (since $(p_1-2)(p_2-2)\geq 3$). This gives us the following result.

\begin{lemma}\label{lem:two}
There is no Zumkeller number of the form $p_1^\alpha p_2^\beta$ where $p_1$ and $p_2$ are distinct odd primes and $\alpha$ and $\beta$ are natural numbers.
\end{lemma}

 So, a Zumkeller number with two distinct prime factors must be even. We characterize all such numbers in the remainder of this section. In the following, we always take $\alpha, \beta \in \mathbb{N}$ and $p$ to be a prime unless otherwise mentioned. We first find necessary and sufficient conditions for a positive integer $n=2^{\alpha}p$ to be a Zumkeller
number and then we extend our results for $n=2^{\alpha}p^{\beta}$. It is easy to see each prime number $p$ can be represented in the unique form
\begin{equation}\label{eq:p}
    p=1+2^{r_1}+2^{r_2}+2^{r_3}+\cdots+2^{r_l},
\end{equation}
where $r_1<r_2<r_3<\cdots <r_l$ for some $l$. Alternatively, this is the base $2$ representation of $p$. In the remainder of this paper, we let $\theta(A)=\sum_{i=1}^k a_i $ for a set $A=\{a_1,a_2,\ldots,a_k\}.$ We are now ready to prove our first results.

\begin{theorem}\label{Thm1}
Let $p$ be a prime number of the form \eqref{eq:p}. The positive integer $n=2^{\alpha}p$ is a Zumkeller number if and only if $\alpha \geq r_l$.
\end{theorem}

 \begin{proof}
 Let $n$ be a Zumkeller number and let the set of all positive divisors of $n$ be
 \[D=\{1,2,2^2,\ldots,2^{\alpha},p,2p,2^2p,\ldots,2^{\alpha}p\}.\]
Put $A=\{2^{\alpha}p\}$ and $B=\{p,2p,\ldots,2^{\alpha-1}p\}$. We have
 \[\theta(B)=p+2p+\cdots+2^{\alpha-1}p=p(1+2+\cdots+2^{\alpha-1})=p(2^{\alpha}-1)=2^{\alpha}p-p.\]
If $r_l>\alpha$, it is impossible to represent the prime number $p$
as a sum of the elements in the set
$C=\{1,2,2^2,\ldots,2^{\alpha}\}$, so we cannot partition the set
$D$ into two disjoint sets since the set which will contain
$2^{\alpha}p$ will always have it's sum of elements greater than
the other. This is a contradiction.

For the converse, let $r_l\leq \alpha$, and $n=2^{\alpha}p$. Consider $A=\{2^{\alpha}p\}$, $B=\{p,2p,\ldots,2^{\alpha-1}p\}$ and $C=\{1,2,2^2,\ldots,2^{\alpha}\}$. Now,
 \begin{eqnarray*}
\dfrac{\sigma(n)}{2}-\theta(B)&=&\dfrac{\sigma(n)}{2}-\dfrac{2p(2^{\alpha}-1)}{2}\\
&=&\dfrac{2^{\alpha+1}p+2^{\alpha+1}-p-1-2^{\alpha+1}p+2p}{2}\\
&=&\dfrac{2^{\alpha+1}+p-1}{2}=2^{\alpha}+\dfrac{p-1}{2}.
 \end{eqnarray*}
  By substituting equation \eqref{eq:p}, we have
 \begin{eqnarray*}
 2^{\alpha}+\dfrac{p-1}{2}=2^{\alpha}+2^{r_1-1}+2^{r_2-1}+\cdots+2^{r_l-1}.
 \end{eqnarray*}
 Since $\alpha\geq r_l$, then, the set $C$ must be contain the numbers $2^{r_1-1}, 2^{r_2-1},\ldots, 2^{r_l-1}$ and $2^{\alpha}$.
  Put
 \[B^{\prime}=B\cup\{2^{\alpha},2^{r_1-1},2^{r_2-1},\cdots,2^{r_l-1}\}.\] Clearly, $\theta(B^{\prime})=\sigma(n)/2$. Now, we add the remaining elements of the set $C$ into the set $A$, which leads to $\theta(A)=\sigma(n)/2$. So, we have partitioned the divisors of $n=2^{\alpha}p$ into two disjoint set $A$ and $B$ with the same sum $\sigma(n)/2$. Hence, $n$ is a Zumkeller number.
 \end{proof}
 Bhakara Rao and Peng \cite[Proposition 2]{Pen} proved that the necessary condition for $n$ to be a Zumkeller number is $\sigma(n)\geq 2n$. The following theorem, shows this condition is necessary and sufficient for $n=2^{\alpha}p$.
 \begin{theorem}\label{Thmo}
 Let $p$ be a prime number. Then $n=2^{\alpha}p$ is a Zumkeller number if and only if $\sigma(n)\geq 2n$.
 \end{theorem}
 \begin{proof}
 The necessary condition follows from Bhaskara Rao and Peng \cite[Proposition 2]{Pen}.

 Let $\sigma(n)\geq 2n$, then
 \begin{eqnarray*}
 \sigma(2^{\alpha}p)\geq 2^{\alpha+1}p\Longrightarrow (2^{\alpha+1}-1)(p+1)\geq 2^{\alpha+1}p\Longrightarrow 2^{\alpha+1}\geq p+1.
 \end{eqnarray*}
 Now, by substituting equation \eqref{eq:p}, we get
 \begin{eqnarray}\label{SS}
  2^{\alpha}\geq 1+2^{r_1-1}+2^{r_2-1}+\cdots+2^{r_l-1}.
 \end{eqnarray}
  By adding $2^{\alpha}$ to both side of the equation \eqref{SS}, we have
  \begin{eqnarray*}
  2^{\alpha+1}\geq 2^{\alpha}+2^{r_1-1}+2^{r_2-1}+\cdots+2^{r_l-1}+1.
 \end{eqnarray*}
 Hence
  \begin{eqnarray*}
  \sigma(2^{\alpha})\geq 2^{\alpha}+2^{r_1-1}+2^{r_2-1}+\cdots+2^{r_l-1}.
 \end{eqnarray*}
 
 Let $D=\{1,2,2^2,\ldots,2^{\alpha},p,2p,2^2p,\ldots,2^{\alpha}p\}$ be a set of all divisors of $n=2^{\alpha}p$. Put  $A=\{2^{\alpha}p\}$, $B=\{p,2p,\ldots,2^{\alpha-1}p\}$ and $C=\{1,2,2^2,\ldots,2^{\alpha}\}$. Obviously, $\theta(B)=2^{\alpha}p-p$ and
 \[\dfrac{\sigma(n)}{2}-\theta(B)=2^{\alpha}+2^{r_1-1}+2^{r_2-1}+\cdots+2^{r_l-1}.\]
 Now, \[\theta(C)\geq 2^{\alpha}+2^{r_1-1}+2^{r_2-1}+\cdots+2^{r_l-1},\] and the set $C$ contains all factors $2^{\alpha},2^{r_1-1},2^{r_2-1},\cdots,2^{r_l-1}$. Put
 \[B^{\prime}=B\cup\{2^{\alpha},2^{r_1-1},2^{r_2-1},\cdots,2^{r_l-1}\}.\]
 Hence, $\theta(B^{\prime})=\sigma(n)/2$.  Now, by adding the remaining elements of the set $C$ into the set $A$, we partitioned all divisors of $n$ into two disjoint sets $A$ and $B^{\prime}$ with the same sum. Hence $n$ is a Zumkeller number.
 \end{proof}
 We can prove the result of Euclid - Euler as a corollary of our results.
 \begin{cor}[Euclid - Euler]
Let $p=1+2+2^2+2^3+\cdots+2^{\alpha}$ be a prime number. Then $n=2^{\alpha}p$ is a perfect number.
 \end{cor}

 \begin{cor}\label{Corr}
 Let $2^{\alpha}<p<2^{\alpha+1}$ be a prime number. Then $n=2^{\alpha}p$ is a Zumkeller number.
 \end{cor}
 
 \noindent Both the corollaries follow directly from Theorem \ref{Thm1}.
 
 We can now look at the more general case when $n=2^\alpha p^\beta$, where $\beta>0$. The proofs in this case are a bit more involved, and our main result of this section is now stated below.

 \begin{theorem}\label{Thm4}
Let $n=2^{\alpha}p^{\beta}$ be a positive integer. Then $n$ is a Zumkeller number if and only if $p\leq 2^{\alpha+1}-1$ and $\beta$ is an odd number.
\end{theorem}
\begin{proof}
Obviously if $n=2^{\alpha}p^{\beta}$ is a Zumkeller number then $\sigma(n)$ must be even. Clearly $\beta$ must be an odd number, since
\[\sigma(n)=\sigma(2^{\alpha}p^{\beta})=\dfrac{2^{\alpha+1}-1}{2-1}\times(1+p+p^2+\cdots+p^{\beta})\]
is even, which means $(1+p+\cdots + p^\beta)$ must be even. Hence $\beta$ must be an odd number.

It is sufficient to show $p\leq 2^{\alpha+1}-1$. We have $\sigma(n)\geq 2n$, then
 \[(2^{\alpha+1}-1)\left(\dfrac{p^{\beta+1}-1}{p-1}\right)\geq 2^{\alpha+1}p^{\beta}.\]Hence
 \[2^{\alpha+1}p^{\beta}-2^{\alpha+1}\geq p^{\beta+1}-1.\] After some simple manipulation, from the above inequality we get
 \[
 2^{\alpha+1}-1\geq \frac{p^{\beta+1}-p^\beta}{p^\beta -1} = p-\frac{p^\beta -p}{p^\beta -1},
 \]
which leads to $2^{\alpha+1}-1\geq p$.

Now, let $\beta$ be an odd number and $2^{\alpha+1}-1\geq p$. We shall prove that $n$ is a Zumkeller number.

Let $D$ be the set of divisors of $2^\alpha p$, and let $D=A\cup B$ be the Zumkeller partition for $2^{\alpha}p$ as in the proof of Theorem \ref{Thm1}. Using these sets, we create a
new partitions for all divisors of
$n=2^{\alpha}p^{\beta}$.

 Let $D^\prime$ be a set of all divisors $n=2^{\alpha}p^{\beta}$ and for each $1\leq i\leq \dfrac{\beta -1}{2}$ consider $A_i=p^{2i}A$ and $B_i=p^{2i}B$. It is easy to see that
\[(A\cup B) \cup (\bigcup_{i=1}^{ \frac{\beta -1}{2}}(A_i\cup B_i))=D^\prime.\]
Put $A^{\prime}=A\cup A_1\cup\cdots\cup A_i$ and $B^{\prime}=B\cup B_1\cup \cdots\cup B_i$ with $1\leq i\leq \dfrac{\beta -1}{2}$. Clearly the sets $A^{\prime}$ and $B^{\prime}$ are disjoint subsets of $D^\prime$ such that $D^\prime=A^{\prime}\cup B^{\prime}$ and $\theta(A^{\prime})=\theta(B^{\prime})$.  This completes our proof.
 \end{proof}
 
 An easy corollary of the above result which connects it to Theorem \ref{Thm1} is the following.
 
 \begin{cor}
 $2^{\alpha}p$ is a Zumkeller number if and only
if
 $2^{\alpha}p^{\beta}$ is a Zumkeller number, where $\beta$
is an odd number.
 \end{cor}

\begin{proof}
Let $2^{\alpha}p$ be a Zumkeller number. Then Theorem \ref{Thm1}
implies $p\leq 2^{\alpha+1}-1$. Since $\beta$ is an odd number, we
get $2^{\alpha}p^{\beta}$ is a Zumkeller number.

Conversely let $2^{\alpha}p^{\beta}$ be a Zumkeller number. Then
Theorem \ref{Thm4} implies $p\leq 2^{\alpha+1}-1$. Therefore
$2^{\alpha}p$ is a Zumkeller number.
\end{proof}

 The following fact from Bhaskara Rao and Peng \cite{Pen} and Clark \textit{et. al.} \cite{Cla} gives a method of generating new Zumkeller numbers from
a known Zumkeller number.
 \begin{lemma}[Bhaskara Rao - Peng]\cite[Theorem 4]{Pen}\label{Peng}
 If $n$ is a Zumkeller number and $p$ is a prime with $(n, p) = 1$, then $np^l$ is Zumkeller for any positive integer $l$.
 \end{lemma}
  For $1\leq i\leq \ell$, let $\beta_i$ be odd numbers and let $p_1=2<p_2=3<p_3=5<\cdots$ be the sequence of prime numbers. According to Lemma \ref{Peng}, we can generalize Theorem \ref{Thm4} for each $n=2^{\alpha}p_1^{\beta_1}\cdots p_{\ell}^{\beta_{\ell}}$ since
$(2^{\alpha}p_i^{\beta_i},\prod_{j=1,i\neq j}^{\ell}p_j^{\beta_j})=1$. Before doing so, we need the following consequence of Bertrand's postulate, the proof of which is left for the reader.

 \begin{lemma}\label{Pri}
 Let $p_1=2<p_2=3<p_3=5<\cdots$ be the sequence of prime numbers. Then $p_i<2^i$ for each $i>1$.
 \end{lemma}

 \begin{theorem}
  Let $p_1=2<p_2=3<p_3=5<\cdots$ be the sequence of prime numbers. Then $2^ip_i$ is a Zumkeller
  number for each $i>1$.
 \end{theorem}
 \begin{proof} 
 From the above lemma we get $p_i<2^i$ for each $i>1$.
Therefore $p_i<2^{i+1}-1$ for each $i>1$. Theorem \ref{Thm1} implies $2^ip_i$ is a Zumkeller number for
each $i>1$.
 \end{proof}
 
We can prove a result similar to Theorem \ref{Thmo} for the more general case, which is done below.

\begin{theorem}
Let $n=2^{\alpha}p^{\beta}$ be a positive integer.  If
$\beta$ is odd, then $n$ is a Zumkeller number if and only if
$\sigma(n)\geq 2n$.
\end{theorem}

\begin{proof}
The necessary condition is trivial, as it has been remarked already that $\sigma(n)\geq 2n$ is a necessary condition for any $n$ to be a Zumkeller number.

Suppose $n$ is not Zumkeller, then by Theorem \ref{Thm4} we have $2^{\alpha+1}<p+1$. Let, $p=2^{\alpha+1}-1+x$, for some positive integer $x$. Obviously
$x>1$ and we have $$\sigma(n)-2n=\dfrac{(2^{\alpha+1}-1)(p^{\beta+1}-1)-2^{\alpha+1}p^{\beta}(p-1)}{p-1}.$$ Now,
$(2^{\alpha+1}-1)(p^{\beta+1}-1)-2^{\alpha+1}p^{\beta}(p-1)=p^{\beta}(1-x)+(x-p)$. Since $1-x<0$ and $x-p<0$, the above expression is negative. Hence $\sigma(n)<2n$.

Again let $\sigma(n)<2n$ and $p=2^{\alpha+1}-1+x$. Then
$p^{\beta}(1-x)+(x-p)<0$. Which implies
$p^{\beta}-p<xp^{\beta}-x$, and therefore
$\dfrac{(p^{\beta-1}-1)p}{p^{\beta}-1}<x$. Therefore $x$ is
positive. which implies $p>2^{\alpha+1}-1$. 

Hence $p>2^{\alpha+1}-1$ if and only if $\sigma(n)<2n$. Therefore $p\leq 2^{\alpha+1}-1$ if and only if $\sigma(n)\geq
2n$. The proof is concluded by Theorem \ref{Thm4}.
\end{proof}

We defer some discussion on Zumkeller numbers with more than two distinct prime factors to Subsection \ref{sub:three}.

\section{Characterization of $k$-layered numbers}\label{sec:layer}

Recently, Jokar \cite{Jokar} studied a further generalization of Zumkeller numbers which he called $k$-layered numbers. A positive integer is said to be a $k$-layered number if its positive divisors can be partitioned into $k$ disjoint subsets of equal sum. Such a partition is then called a $k$-partition. We can see immediately that Zumkeller numbers are $2$-layered. 

The aim of this section is to characterize all $k$-layered numbers with two distinct odd prime divisors and extend some results to the more general cases. We shall use the facts that if $n$ is a $k$-layered number, then $k|\sigma(n)$ and $\sigma(n)\geq kn$.

\begin{proposition}
$p^{\alpha}$ is not a $k$-layered number for any $k\geq 3$
and for any prime $p$.
\end{proposition}

\begin{proof}
Let $n=p^{\alpha}$. Then $\sigma(n)\geq kn$ implies $(p-1)(1-k)\geq 0$ after some algebraic manipulation. Clearly $p-1>0$, so $1\geq k$ which is not possible.
\end{proof}

\begin{theorem}
For any distinct primes $p$ and $q$, $p^{\alpha}q^{\beta}$ is not a $k$-layered number for any $k\geq 3$.
\end{theorem}

\begin{proof}
Let $n=p^{\alpha}q^{\beta}$. Then $\sigma(n)\geq kn$
implies $\left(\dfrac{p^{\alpha}-1}{p^{\alpha}(p-1)}+1\right)\left(\dfrac{q^{\beta}-1}{q^{\beta}(q-1)}+1\right)\geq
k$ after some algebraic manipulation. This gives us \[\left\{\dfrac{1}{p-1}\left(1-\dfrac{1}{p^{\alpha}}\right)+1\right\}\left\{\dfrac{1}{q-1}\left(1-\dfrac{1}{q^{\beta}}\right)+1\right\}\geq
k.\] Since $\dfrac{1}{p-1}\left(1-\dfrac{1}{p^{\alpha}}\right)$ and
$\dfrac{1}{q-1}\left(1-\dfrac{1}{q^{\beta}}\right)$ are less than $1$, therefore $k$ is at most $4$.

Let $k=3$, then $\sigma(n)\geq 3n$
implies
$\left(\dfrac{p^{\alpha+1}-1}{p-1}\right)\left(\dfrac{q^{\beta+1}-1}{q-1}\right)\geq
3p^{\alpha}q^{\beta}$, which gives us
\begin{equation}\label{pjm:a}
    p^{\alpha}q^{\beta}(-2pq+3p+3q-3)-p^{\alpha+1}-q^{\beta+1}+1\geq
0
\end{equation}

Now let $q=p+x$, where $x\geq 1$. Then
$$-2pq+3p+3q-3=x(3-2p)+2p(3-p)-3.$$ Obviously
this is negative for all $p>2$. If $p=2$, then $-2pq+3p+3q-3=-x+1$. This is negative for all
$x>1$. If $p=2$ and $x=1$, then $-2pq+3p+3q-3=0$. So in all cases the L.H.S. of the inequality \eqref{pjm:a} is negative. Thus we
get a contradiction and $k\neq 3$.

The case for $k=4$ is exactly similar and we can easily prove that it is not possible.
\end{proof}

One of the main results of this section is the following result.

\begin{theorem}\label{63}
If $n=2^\alpha p q$ where $\alpha\geq 1$ and $p$ and $q$ are distinct odd primes, then $n$ is not a $k$-layered number for $k>3$.

Furthermore, if $k=3$ then $n$ is a $3$-layered number only when $n=15.2^\alpha$ with $\alpha\geq 3$ or when $n=21.2^\alpha$ with $\alpha\geq 5$ and $\alpha$ is odd.
\end{theorem}

\begin{proof}
We have, $\sigma(n)\geq kn$ implies
 \begin{equation}\label{pjmx}
     (2^{\alpha+1}-1)(p+1)(q+1)\geq k2^{\alpha}pq
 \end{equation}
 which gives us $\left(2-\dfrac{1}{2^{\alpha}}\right)\left(1+\dfrac{p+q+1}{pq}\right)\geq k$. Since $p$ and $q$ are primes, so $\dfrac{p+q+1}{pq}\leq 1$. Hence $4>k$. This proves the first part.

Now let $k=3$. Then $\sigma(n)\geq 3n$ implies $2^{\alpha}(-pq+2p+2q+2)-(pq+p+q+1)\geq 0$, and clearly we must have $-pq+2p+2q+2>0$. Let $q=p+x$, where $x\geq 2$. Then we get $-p^2-px+2p+2q+2x+2>0$
\begin{equation}\label{pjmxi}
    p^2-4p<2x-px+2,
\end{equation}
which gives us $p(p-5)<(2-p)(x+1)$.

The R.H.S. of this inequality is negative, so the L.H.S. must be negative. Hence $p-5<0$ which implies $p=3$. Therefore \eqref{pjmxi} implies $9-12<2x-3x+2$
 which gives $x<5$ and so $q=5$ or $7$.

If $q=5$, then \eqref{pjmx} implies $(2^{\alpha+1}-1)4.6\geq
3.2^{\alpha}.3.5$. After some algebraic manipulation we get from this $\alpha\geq 3$. If $q=7$, then \eqref{pjmx} implies $(2^{\alpha+1}-1)4.8\geq
3.2^{\alpha}.3.7$ from which we get $\alpha\geq 5$. So, we have the two cases which we need to work with now.

\textbf{Case 1:} $n=2^{\alpha}.3.5$, where $\alpha\geq 3$.

If $n$ is a $k$-layered number, then it is necessary that
$k|\sigma(n)$. Here $\sigma(n)=(2^{\alpha+1}-1)4.6$, which is
divisible by 3. Let $\alpha>3$.

Now 
\begin{align*}
    \dfrac{\sigma(n)}{3}-n=&~2^\alpha -8\\
    =&~ (3+5)(1+2+2^2+\dots+2^{\alpha-4}).
\end{align*}
Therefore
$$\dfrac{\sigma(n)}{3}=n+(3+5)(1+2+2^2+\dots+2^{\alpha-4}).$$ Let
$A=\{n,3,2.3,2^2.3,\dots,2^{\alpha-4}.3,5,2.5,2^2.5,\dots,2^{\alpha-4}.5\}$

Again
\begin{align*}
    \dfrac{\sigma(n)}{3}=&~8(2^{\alpha+1}-1)\\
    =&~8(1+2+2^2+\dots+2^{\alpha})\\
    =&~8+16(1+2+2^2+\dots+2^{\alpha-1})\\
    =&~8+(1+2+2^2+\dots+2^{\alpha-1})+15(1+2+2^2+\dots+2^{\alpha-1})\\
    =&~8+2^{\alpha}-1+15(1+2+2^2+\dots+2^{\alpha-1})\\
    =&~1+2+2^2+2^{\alpha}+3.5(1+2+2^2+\dots+2^{\alpha-1}).
\end{align*}
Let
$B=\{1,2,2^2,2^{\alpha},3.5,2.3.5,2^2.3.5,\dots,2^{\alpha-1}.3.5\}$

Now let $C=D-A\cup B$, where $D$ is the set of all divisors of $n$. Then $\theta(C)=\dfrac{\sigma(n)}{3}$ and we get that $\{A,B,C\}$ is the $3$-partition of $n$.

When $\alpha=3$ then $n=120$ and we can take the sets $A, B, C$ to be the following
\[A=\{1,2,4,8,15,30,60\}, \quad B=\{3,5,6,10,12,20,24\} \]and $C=\{120\}$, which gives us the $3$-partition.

\textbf{Case 2:} $n=2^{\alpha}.3.7$, where $\alpha\geq 5$.

Here $\sigma(n)=(2^{\alpha+1}-1)4.8$. Therefore $3|\sigma(n)$ if and only if $3|(2^{\alpha+1}-1)$ if and only if $\alpha$ is odd.

Now
\begin{align*}
  \dfrac{\sigma(n)}{3}-n=&~\dfrac{32(2^{\alpha+1}-1)}{3}-21.2^{\alpha} \\
  =&~\dfrac{2^{\alpha}-32}{3}\\
  =&~\dfrac{2(2^{\alpha-1}-1)-30}{3}\\
  =&~\dfrac{2(1+2+2^2+\dots+2^{\alpha-2})-30}{3}\\
  =&~\dfrac{2.3(1+2^2+2^4+\dots+2^{\alpha-3})-30}{3}\\
  =&~2(1+2^2+2^4+\dots+2^{\alpha-3})-10\\
  =&~2+2^3+2^5+2^7+\dots+2^{\alpha-2}-10.
\end{align*}
So $$\dfrac{\sigma(n)}{3}=n+2^5+2^7+\dots+2^{\alpha-2}.$$ Let $A=\{n,2^5,2^7,\dots,2^{\alpha-2}\}$.

Again
\begin{align*}
    \dfrac{\sigma(n)}{3}=&~\dfrac{32(2^{\alpha+1}-1)}{3}\\
    =&~\dfrac{32(1+2+2^2+\dots+2^{\alpha})}{3}\\
    =&~32(1+2^2+2^4+\dots+2^{\alpha-1})\\
    =&~(21+7+3+1)(1+2^2+2^4+\dots+2^{\alpha-1})\\
    =&~3.7(1+2^2+2^4+\dots+2^{\alpha-1})\\&~~+7(1+2^2+2^4+\dots+2^{\alpha-1})+3(1+2^2+2^4+\dots+2^{\alpha-1})\\&~~+(1+2^2+2^4+\dots+2^{\alpha-1}).
\end{align*}
Let $B=\{1,2^2,\ldots, 2^{\alpha-1},3,3.2^2,\ldots, 3.2^{\alpha-1},7,7.2^2, \ldots, 7.2^{\alpha-1}, 3.7, 3.7.2^2, \ldots, 3.7.2^{\alpha-1}\},$ and let $C=D-A\cup B$, where $D$ is the set of all divisors of $n$. Then,
$\theta(C)=\dfrac{\sigma(n)}{3}$ and we get that $\{A,B,C\}$ is the $3$-partition.

This completes the proof.
\end{proof}

\begin{remark}
$2^3.3.5=120$ is the smallest $k$-layered number, where $k\geq
3$.
\end{remark}

The following result of Jokar \cite{Jokar} is needed for our next results.
\begin{proposition}\label{64}
\cite[Proposition 2.3]{Jokar}
If $n$ is a
$k$-layered number and $(n,w)=1$ then $nw$ is a $k$-layered number.
\end{proposition}
\noindent Now, combining Proposition \ref{64} with Theorem \ref{63} we get the following result.
\begin{theorem}
\begin{enumerate}
    \item $2^{\alpha}.3.5.{p_4}^{r_4}{p_5}^{r_5}\dots {p_l}^{r_l}$ is a
$3$-layered number for all $l,r_4,r_5,\dots,r_l$ and $\alpha\geq 3$.
\item $2^{\alpha}.3.5^{\beta}.7.{p_5}^{r_5}{p_6}^{r_6}\dots
{p_l}^{r_l}$ is a $3$-layered number for all
$l,\beta,r_5,r_6,\dots,r_l,\alpha\geq 5$ and $\alpha$ odd.
\end{enumerate}
\end{theorem}

An easy result for $k$-layered numbers with more than $3$ distinct prime factors is the following.
\begin{theorem}\label{layer:three}
Let $n=p_1p_2\dots p_l$, where $l\leq 15$ and $p_i$s are odd distinct primes. Then $n$ is not a $k$-layered number for any $k\geq 3$.
\end{theorem}
\begin{proof}
This follows from the elementary inequality $\prod_{i=1}^{15}\left(1+\dfrac{1}{p_i}\right)<3$.
\end{proof}

We will close this section with some results which connects $k$-layered numbers with Zumkeller numbers and practical numbers. Recall that $n$ is called a practical number if all smaller integers than $n$ can be written as sums of distinct divisors of $n$. Before going into our results we need the following facts.
\begin{proposition}\label{prac-1}\cite{sri}
If $n$ is a practical number, then $\sigma(n)\geq 2n-1$.
\end{proposition}

\begin{proposition}\cite[Proposition 8]{Pen}\label{prac-2}
A positive integer $n$ is a practical number if and only if every positive integer $\leq \sigma(n)$ is a sum of distinct positive divisors of $n$.
\end{proposition}

\begin{theorem}
$n$ is a $2k$-layered number, where $k>1$. Then $n$ is Zumkeller
number.
\end{theorem}
\noindent This follows quite easily from the definitions. 

\begin{theorem}\label{thm-p}
Let $n\geq 2$ be a practical number such that $3|\sigma(n)$. Let $x$ be
a Zumkeller number such that $(n,x)=1$. Then $nx$ is $3$-layered
number.
\end{theorem}

\begin{proof}
We have $\sigma(n)\geq 2n-1$ and $\sigma(x)\geq 2x$. Now $\sigma(nx)=\sigma(n)\sigma(x)$, since $(n,x)=1$. So, $\sigma(nx)\geq (2n-1)(2x)\geq 3nx$.

Now $\dfrac{\sigma(nx)}{3}=\dfrac{\sigma(n)}{3}.\sigma(x)$. By Proposition \ref{prac-2}, $\dfrac{\sigma(n)}{3}$ is sum of distinct divisors of $n$. We have,
$\sigma(nx)=\dfrac{\sigma(n)}{3}.\sigma(x)+\dfrac{2\sigma(n)}{3}.\sigma(x)=\dfrac{\sigma(n)}{3}.\sigma(x)+\dfrac{2\sigma(n)}{3}.\dfrac{\sigma(x)}{2}+\dfrac{2\sigma(n)}{3}.\dfrac{\sigma(x)}{2}$.

Let the set of all divisors of $n$ be $D_n$ and that of $x$ be
$D_x$. Let the set of the distinct divisors of $n$ for which
$\dfrac{\sigma(n)}{3}$ can be expressed as a sum of distinct
divisors of $n$ be $A$. Let $A'=A.D_x$. Since $x$ is a Zumkeller number, let the partition be $\{B,C\}$. Let $B'=(D_n-A).B$ and $C'=(D_n-A).C$. Then $\{A',B',C'\}$ is the partition for $nx$. That is $nx$ is
$3$-layered number. 
\end{proof}

\begin{exa}
$x=945=3^3.5.7$ is the smallest odd Zumkeller number. 2, 8, 32,
88, 104 are the first five practical numbers $n$ such that
$3|\sigma(n)$ and $(x,n)=1$. Therefore 2.945, 8.945, 32.945,
88.945, 104.945 are 3-layered numbers. That is 1890, 7560, 30240,
83160, 98280 are 3-layered numbers.
\end{exa}

We can generalize Theorem \ref{thm-p} further.

\begin{theorem}
Let $n\geq 2$ be a practical number and $k|\sigma(n)$ where $k\geq
3$. Let $x$ be a $(k-1)$-layered number such that $(n,x)=1$. Then
$nx$ is $k$-layered number.
\end{theorem}

\begin{proof}
Since, $\sigma(n)\geq 2n-1$ and $\sigma(x)\geq (k-1)x$, we get $\sigma(nx)=\sigma(n)\sigma(x)\geq knx$.

Now
\begin{align*}
    \sigma(nx)=&~\sigma(n)\sigma(x)=\dfrac{\sigma(n)}{k}.\sigma(x)+\dfrac{(k-1)\sigma(n)}{k}.\sigma(x)\\
    =&~\dfrac{\sigma(n)}{k}.\sigma(x)+\dfrac{(k-1)\sigma(n)}{k}.\left\{\underbrace{\dfrac{\sigma(x)}{k-1}+\dfrac{\sigma(x)}{k-1}+\dots+\dfrac{\sigma(x)}{k-1}}_{k-1}\right\}.
\end{align*}
The rest of the proof works in a similar way to that of Theorem \ref{thm-p}.
\end{proof}

\section{Some further results on Zumkeller numbers}

The aim of this section is to show some results connected with Zumkeller numbers and Harmonic numbers as well as to hint at some other type of results that can be derived for Zumkeller numbers. We hope that others would take up this direction of research in the near future.

 \subsection{Zumkeller numbers and Harmonic numbers}\label{sec:harmonic}

 In 1948, Ore \cite{Ore} introduced the concept of \textit{Harmonic numbers}, and these numbers were named as \textit{Ore’s harmonic numbers} by Pomerance \cite{Pom}. Let $n$ be a positive integer, the harmonic mean of the divisors of $n$ is defined as
\[\dfrac{1}{H(n)}=\dfrac{1}{\tau(n)}\sum_{d|n}\dfrac{1}{d}.\]
And since
\[n\sum_{d|n}\dfrac{1}{d}=\sum_{d|n}\dfrac{n}{d}=\sum_{d|n}d=\sigma(n),\]
it follows that
\[H(n)=\dfrac{n\tau(n)}{\sigma(n)}.\]
Therefore, we remark that $H(n)$ is an integer if and only if $\sigma(n)|n\tau(n)$. A number $n$ satisfying the condition  $\sigma(n)|n\tau(n)$ is called a harmonic number. Obviously $H(n)>1$ and if $m=2^{n-1}(2^n-1)$ is a perfect number then $H(m)=n$. Since
\[H(m)=\dfrac{m\tau(m)}{\sigma(m)}=\dfrac{2\times m\times n}{2m}=n.\]

Harmonic numbers are interesting in their own rights and they have some connection with perfect numbers. A result in this direction was given by Laborde \cite{Lab}.
\begin{theorem}[Laborde, \cite{Lab}]\label{Lab}
If a given integer $n$ is even and has the form
\[n=2^{H(n)-1}(2^{H(n)}-1).\]
Then $n$ must be a perfect number.
\end{theorem}
\noindent We can prove another very simple result of this type.
\begin{lemma}
If $n=2^{H(n)-1}(2^{H(n)}-1)$ is even, then $H(2^{H(n)}-1)<2$.
\end{lemma}

\begin{proof}
We have
\begin{eqnarray*}
H(n)&=&\dfrac{2^{H(n)-1}\times (2^{H(n)}-1)\times \tau(2^{H(n)-1})\times \tau(2^{H(n)}-1)}{\sigma(2^{H(n)-1})\times \sigma(2^{H(n)}-1)}\\
&=&\dfrac{2^{H(n)-1}\times (2^{H(n)}-1)\times H(n)\times \tau(2^{H(n)}-1)}{(2^{H(n)}-1)\times \sigma(2^{H(n)}-1)}
\end{eqnarray*}
Therefore, $1=\dfrac{2^{H(n)-1}\times
\tau(2^{H(n)}-1)}{\sigma(2^{H(n)}-1)}$. Also, we have
 \[H(2^{H(n)}-1)=\dfrac{(2^{H(n)}-1)\times \tau(2^{H(n)}-1)}{\sigma(2^{H(n)}-1)}.\]Hence
 \[H(2^{H(n)}-1)=\dfrac{(2^{H(n)}-1)}{2^{H(n)-1}}=2-\dfrac{1}{2^{H(n)-1}}.\]
 This completes our proof.
\end{proof}

In the remainder of this subsection we discuss some connections of Zumkeller numbers with Harmonic numbers.

\begin{proposition}
\label{Thm2}
Let $n$ be a Zumkeller number, then $H(n)\leq \dfrac{\tau(n)}{2}.$
Furthermore, if $n$ is a $k$-layered number, then $H(n)\leq \dfrac{\tau(n)}{k}.$
\end{proposition}

\begin{proof}
Since $n$ is a Zumkeller number so $\sigma(n)\geq 2n$. Therefore
\[H(n)=\dfrac{n\tau(n)}{\sigma(n)}\leq \dfrac{n\tau(n)}{2n}.\]
Hence  $H(n)\leq \dfrac{\tau(n)}{2}.$ The proof for the second part is analogous.
\end{proof}

\begin{cor}\label{Cor2}
$n=2^{\alpha}p$ is a Zumkeller number if and only if $H(n)\leq \dfrac{\tau(n)}{2}.$ Also, equality holds when $n$ is a perfect number.
\end{cor}
\begin{proof}
Using Proposition \ref{Thm2}, if $n$ is a Zumkeller number, then $H(n)\leq \dfrac{\tau(n)}{2}.$ For converse, we have
\[H(n)=\dfrac{n\tau(n)}{\sigma(n)}\leq\dfrac{\tau(n)}{2}.\]
Thus $\sigma(n)\geq 2n$. The proof is concluded by Theorem \ref{Thmo}.
\end{proof}
\begin{cor}
Let $n>6$ be a positive integer. For distinct prime numbers $p$ and $q$ with $n=pq$ is not a Zumkeller number.
\end{cor}
\begin{proof}
If $p$ and $q$ are distinct odd primes, then $n=pq$ is not
a Zumkeller number. This was discussed in the beginning of Section \ref{sec:two}. Now let $n=2p$, where $p$ is an odd prime. Then
$\tau(n)=4,\sigma(n)=3(p+1)$ and $H(n)=\dfrac{8p}{3(p+1)}=2+\dfrac{2}{3}-\dfrac{8}{3(p+1)}$.

But, $\dfrac{2}{3}-\dfrac{8}{3(p+1)}>0$ if and only if $p>3$. Therefore $H(n)>2$
for all $p>3$. That is $H(n)>\dfrac{\tau(n)}{2}$ for all $p>3$.
Hence $n$ is not a Zumkeller number for all $p>3$.
\end{proof}

The following is an almost immediate consequence of the above results.

\begin{cor}
If $n$ is a prime number or a semi-prime number\footnote{Numbers with only two prime factors.},
then $n$ is not a Zumkeller number except $n=6$.
\end{cor}

Although conditions of the following lemma are weaker than
Corollary \ref{Corr}, but it is a more elegant proof using
Harmonic mean numbers.

\begin{lemma}
For a prime $p<2^{\alpha+1}-1$ for some $\alpha$, let
$n=2^{\alpha}(2^{\alpha+1}-1)$ be a perfect number. Then
$m=2^{\alpha}p$ is a Zumkeller number.
\end{lemma}
\begin{proof}
Using Corollary \ref{Cor2}, it is enough to show $H(m)\leq
\dfrac{2(\alpha+1)}{2}$. Let $p$ and $q$ be two prime numbers with $p>q$, then
$H(2^{\alpha}p)>H(2^{\alpha}q)$. If $H(2^{\alpha}p)\leq
H(2^{\alpha}q)$, then
\[H(2^{\alpha}p)=\dfrac{2^{\alpha}p\times (\alpha +1)\times 2}{(2^{\alpha+1}-1)\times (p+1)}\leq\dfrac{2^{\alpha}q\times (\alpha +1)\times 2}{(2^{\alpha+1}-1)\times (q+1)}=H(2^\alpha q).\]
If $\dfrac{p}{p+1}\leq\dfrac{q}{q+1}$,then $pq+p\leq pq+q$ and $p\leq q$, which is a contradiction. If $n$ is a perfect
number, then $H(n)=\alpha+1$ and $2^{\alpha+1}-1$ is a prime.
Since $p<2^{\alpha+1}-1$, we have
 \[H(2^{\alpha}p)<H(2^{\alpha}(2^{\alpha+1}-1)).\]
 Hence $H(2^{\alpha}p)<\alpha+1$. Then $H(2^{\alpha}p)<\dfrac{\tau(2^{\alpha}p)}{2}$ and this
completes our proof.
\end{proof}

\begin{proposition}
\label{Thm3}
Let $n=2^{\alpha}( 2^{\alpha+1}-1)$ be a Zumkeller number. Then \[H(n)<2^{2\alpha+1}.\]
\end{proposition}

\begin{proof}
Proposition \ref{Thm2} leads to $H(n)\leq\dfrac{(\alpha
+1)\tau(2^{\alpha+1}-1)}{2}$. Let
$2^{\alpha+1}-1=p_1^{\alpha_1}p_2^{\alpha_2}\cdots
p_r^{\alpha_r}$ for some primes $p_1, p_2, \ldots, p_r$. For each prime $p$ and $1\neq k\in \mathbf{N}$, we have
$p^{k}>k+1$. (This follows easily from the elementary inequality
$2^k\geq k+1$ for all $k\in \mathbb{N}$). Put $p=2$, then
$\alpha+1<2^{\alpha+1}-1.$ Also, for each positive integer $n>2$,
$\tau(n)<n$.

Consider $n=2^{\alpha+1}-1$, then
$\tau(2^{\alpha+1}-1)<2^{\alpha+1}-1$. By using these facts, we
have
\begin{eqnarray*}
H(2^{\alpha}( 2^{\alpha+1}-1))
&\leq &\dfrac{(\alpha +1)\tau(2^{\alpha+1}-1)}{2}\\
&<&\dfrac{(2^{\alpha+1}-1)^2}{2}<\dfrac{(2^{\alpha+1})^2}{2}=2^{2\alpha+1}.
\end{eqnarray*}
\end{proof}

\subsection{Zumkeller numbers with more than two distinct prime factors}\label{sub:three}

Bhaskara Rao and Peng \cite[Proposition 20]{Pen} provided several bounds on Zumkeller numbers with less than seven distinct prime factors. In the following result we extend these bounds.

\begin{theorem}\label{thm:b}
Let $n=p_1^{a_1}p_2^{a_2}\cdots p_m^{a_m}$ be a Zumkeller number where all $p_i$'s are distinct primes. Then
\begin{enumerate}
    \item If $m=4$, then $p_1=3$. If $p_2=7$, then $p_3\leq 13$.
    \item If $m=5$, then $p_1=3$. If $p_2=11$, then $p_3=13, p_4=17$ and $p_5=19$.
    \item If $m=6$ then $p_1=3$. If $p_2=11$, then $p_3\leq 17$.
    \item If $m\leq 15$ and if $p_1=3$ then $p_2\leq 23$.
    \item If $m\leq 8$ and if $p_1=3$, $p_2=5$, then $p_3\leq 79$.
\end{enumerate}
\end{theorem}

\begin{proof}
The result follows by repeated application of the following elementary inequality
\[
2\leq \dfrac{\sigma(n)}{n}\leq \prod_{i=1}^m\dfrac{p_i}{p_i-1}.
\]
For instance, to see the second bound we notice that
\[
2\leq \dfrac{3}{2}\times \dfrac{11}{10}\times \dfrac{17}{16}\times \dfrac{19}{18}\times \dfrac{23}{22}\leq 1.94,
\]
which is a contradiction, so $p_3=13$ in this case. And similarly we can show the other bounds.
\end{proof}

\begin{remark}
The bounds in Theorem \ref{thm:b} can be extended using the same inequality by taking larger values of $m$. This is one idea which is exploited in finding odd deficient-perfect numbers\footnote{Numbers $n$ such that $\sigma(n)=2n-d$, where $d$ is a proper divisor of $n$.} (see for instance, the work of the second author with Dutta \cite{DuttaSaikia}).
\end{remark}

We can say a bit more about even Zumkeller numbers with two distinct odd prime factors. Bhakara Rao and Peng \cite{Pen} also proved the following theorem.
\begin{theorem}[Bhaskara Rao - Peng]\cite[Corollary 5]{Pen}\label{Bha}
Let $n$ be a Zumkeller number and $(n,m) = 1$ then $nm$ is
a Zumkeller number.
\end{theorem}

\noindent We close this section with the following theorems which uses Theorem \ref{Bha}.
\begin{theorem}
Le $n$ be a Zumkeller number. Then $2n$ is a Zumkeller number.
\end{theorem}
\begin{proof}
Let $n$ be a odd number, then using Theorem \ref{Bha}, $2n$ is a Zumkeller number.

Let $n=2^{\alpha}\ell$ with $(2,\ell)=1$ and $D$ be the set of all divisors of $n$. Since $n$ is a Zumkeller number, it is possible to partition $D$ into two disjoint subsets $A$ and $B$ with the same sum. Let $D^\prime$ and $D^{\prime \prime}$ be set of all positive divisors of $2n$ and $\ell$, respectively. Consider $\Gamma(A)=\{2^{\alpha}d\in A|d\in D^{\prime \prime}\}$ and
$\Gamma(B)=\{2^{\alpha}d\in B|d\in D^{\prime\prime}\}$. Clearly
\[D^\prime=D\cup \{2^{\alpha+1}d|d\in D^{\prime\prime}\}.\]
Now, we put \[A^{\prime}=(A\backslash \Gamma(A)) \cup \{2^{\alpha+1}d|2^{\alpha}d\in A\}\cup \Gamma(B),\]
and
\[B^{\prime}=(B\backslash \Gamma(B)) \cup \{2^{\alpha+1}d|2^{\alpha}d\in B\}\cup
 \Gamma(A).\]
 Since $\theta(A)=\theta(B)$, clearly $A^{\prime}$ and $B^{\prime}$ are a partition of the set $D^\prime$ with $\theta(A^{\prime})=\theta(B^{\prime}).$
\end{proof}

\begin{theorem}
Every $12$ consecutive numbers has one Zumkeller number.
\end{theorem}
\begin{proof}
Every $12$ consecutive numbers must include at least one number which is divisible by $6$, but not by $9$. Let us call this number $n=2^\alpha 3m$ for some $\alpha\geq 1$ and $m\in \mathbb{N}$. Clearly, $(2^{\alpha}3,m)=1$. Then, by Theorem \ref{Bha} $n=2^{\alpha}3m$ is Zumkeller.
\end{proof}
\begin{remark}
There exist $11$ consecutive non Zumkeller numbers $283$ through $293$.
\end{remark}

\section{Concluding Remarks}\label{conc}

We have proved several results characterizing Zumkeller and $k$-layered numbers with two and three distinct prime factors, but the study is by no means complete. The following directions of study appear to us which might lead to some nice results, we leave these as open questions for the readers.
\begin{enumerate}
    \item In Section \ref{sec:two} we proved several criteria for $n=2^\alpha p^\beta$ to be a Zumkeller number, it might be possible to extend these type of results for a general even $n$.
    \item In Section \ref{sec:layer} we have only just touched the surface of results for $k$-layered numbers. A systematic study, like that done by Bhaskara Rao and Peng \cite{Pen} for Zumkeller numbers, if done for $k$-layered numbers would no doubt reveal many more properties.
    \item In Subsection \ref{sec:harmonic}, the connections between harmonic mean numbers with other number sequences could lead to interesting arithmetic properties of Zumkeller numbers.
    \item In Subsection \ref{sub:three}, several of the bounds presented could no doubt be extended much further using more sophisticated analytic techniques.
\end{enumerate}

\section*{Acknowledgements}

The authors thank the anonymous referee for a careful reading of the manuscript and several helpful suggestions.

\bibliographystyle{alpha}
\newcommand{\etalchar}[1]{$^{#1}$}

\end{document}